\documentclass[11pt, twoside]{article}
\usepackage{latexsym}
\usepackage{cite}
\usepackage{amsmath}
\usepackage{amssymb}
\usepackage[all]{xy}
\usepackage{amsfonts}
\usepackage{verbatim}
\usepackage{amsthm}
\usepackage{mathrsfs}
\usepackage{epsfig}
\usepackage{xy}
\usepackage{array}
\usepackage{stmaryrd}
\usepackage{graphicx,color}
\usepackage{xcolor}
\usepackage[colorlinks=true,linkcolor=blue,citecolor=blue]{hyperref}
\usepackage{tikz}
\usetikzlibrary{arrows,calc}
\usepackage{etex}
\usepackage{mathdots}
\usepackage{float}
\usepackage{graphics}
\usepackage{pdflscape}
\usepackage{extarrows}
\usepackage{anysize,hyperref}
\input xypic
\xyoption{all}
\usepackage{bm}
\usepackage[perpage,symbol]{footmisc}
\usepackage{setspace}
\def\C{\mathscr{C}}
\def\E{\mathbb{E}}

\def\s{\mathfrak{s}}
\def\id{\mathrm{id}}
\def\op{^\mathrm{op}}
\def\Ab{\mathsf{Ab}}
\def\del{\delta}
\def\dr{\ar@{->}[r]}

\def\X{\mathscr{X}}
\def\Y{\mathscr{Y}}

\newcommand{\CC}{{\bf{C}}^{n+2}_{\C}}
\newcommand{\mr}{\hbox{\boldmath$\cdot$}}
\newcommand{\ov}{\overset}
\newcommand{\lra}{\longrightarrow}
\newcommand{\co}{\colon}
\newcommand{\uas}{^{\ast}}            
\newcommand{\sas}{_{\ast}}
\newcommand{\Xd}{\langle X^{\mr},\del\rangle}  
\newcommand{\Yr}{\langle Y^{\mr},\rho\rangle}  

\newcommand{\ush}{^\sharp}           
\newcommand{\ssh}{_\sharp}
\SelectTips{cm}{10}

\begin{document}
\baselineskip=15pt
\title{\Large{\bf  A new characterization of $\bm n$-exangulated categories with $ (\bm{n+2})$-angulated structure\footnotetext{\hspace{-1em}This work was supported by the National Natural Science Foundation of China (Grant Nos. 11901190 and 11671221), and the Hunan Provincial Natural Science Foundation of China (Grant No. 2018JJ3205),  and by the Scientific Research Fund of Hunan Provincial Education Department (Grant No. 19B239).}}}
\medskip
\author{Jian He and Panyue Zhou}

\date{}

\maketitle
\def\blue{\color{blue}}
\def\red{\color{red}}

\newtheorem{theorem}{Theorem}[section]
\newtheorem{lemma}[theorem]{Lemma}
\newtheorem{corollary}[theorem]{Corollary}
\newtheorem{proposition}[theorem]{Proposition}
\newtheorem{conjecture}{Conjecture}
\theoremstyle{definition}
\newtheorem{definition}[theorem]{Definition}
\newtheorem{question}[theorem]{Question}
\newtheorem{remark}[theorem]{Remark}
\newtheorem{remark*}[]{Remark}
\newtheorem{example}[theorem]{Example}
\newtheorem{example*}[]{Example}
\newtheorem{condition}[theorem]{Condition}
\newtheorem{condition*}[]{Condition}
\newtheorem{construction}[theorem]{Construction}
\newtheorem{construction*}[]{Construction}

\newtheorem{assumption}[theorem]{Assumption}
\newtheorem{assumption*}[]{Assumption}

\baselineskip=17pt
\parindent=0.5cm

\begin{abstract}
Herschend--Liu--Nakaoka introduced the notion of $n$-exangulated categories. It is not only a higher dimensional analogue of extriangulated categories defined by Nakaoka--Palu, but also gives a simultaneous generalization of $(n+2)$-angulated in the sense of Geiss--Keller--Oppermann and $n$-exact categories in the sense of Jasso.
In this article, we show that an $n$-exangulated category has the structure of an $(n+2)$-angulated category if and only if for any object $X$ in the category, the morphism $0\to X$ is a trivial deflation and the morphism $X\to 0$ is a trivial inflation.\\[0.5cm]
\\[0.1cm]
\medskip
\end{abstract}

\pagestyle{myheadings}
\markboth{\rightline {\scriptsize J. He and P. Zhou }}
         {\leftline{\scriptsize A new characterization of $n$-exangulated categories with $ (n+2)$-angulated structure}}

\section{Introduction}
Triangulated categories and exact categories are two fundamental structures in algebra, geometry and topology.
They are also important tools in many mathematical branches.
Nakaoka and Palu \cite{NP} recently introduced the notion
of extriangulated categories, whose extriangulated structures are given by $\E$-triangles with some axioms. Triangulated categories and exact categories are extriangulated categories. There are a lot of examples of extriangulated categories which are neither triangulated categories nor exact categories, see \cite{NP,ZZ,HZZ1}.
So how to judge whether extriangulated categories are not triangulated categories or exact categories become very important. From \cite{NP}, we already know that exact categories can be regarded as extriangulated categories, whose inflations are
monomorphic and whose deflations are epimorphic. Conversely,
extriangulated categories which are exact categories are those where every
inflation is a monomorphism and every deflation is an epimorphism, see \cite[Corollary 3.18]{NP}. On the other hand, extriangulated categories which are triangulated are those where the ${\rm Ext}^1$ bifunctor $\E$ is such that $\E(-,-) ={\rm Hom}(-,\Sigma-)$ for some auto-equivalence $\Sigma$ on the category, see \cite[Proposition 3.22]{NP}.
Recently,  Msapato \cite[Theorem 3.3]{M} gave another characterization
of extriangulated categories which are triangulated. More specifically, an extriangulated category has the structure of a triangulated category if and only if for any object $X$ in the category, the morphism $0\to X$ is a deflation
and the morphism $X\to 0$ is an inflation.

In \cite{GKO}, Geiss, Keller and Oppermann introduced
a new type of categories, called $(n+2)$-angulated categories, which generalize
triangulated categories: the classical triangulated categories are the special case $n=1$. They appear for example as $n$-cluster tilting subcategories of triangulated categories which are closed under the $n$th power of the shift functor.
Later, Jasso \cite{J} introduced $n$-exact categories which are categories inhabited
by certain exact sequences with $n+2$ terms, called $n$-exact sequences. The case $n=1$
corresponds to the usual concepts of exact categories.
An important source of examples of $n$-exact categories are $n$-cluster
tilting subcategories, see \cite[Theorem 4.14]{J}.
Recently, Herschend, Liu and Nakaoka \cite{HLN} introduced the notion of $n$-exangulated categories. It should be noted that the case $n =1$ corresponds to extriangulated categories. As typical examples we have that $n$-exact and $(n+2)$-angulated categories are $n$-exangulated, see \cite[Proposition 4.5 and Proposition 4.34]{HLN}. However, there are some other examples of $n$-exangulated categories which are neither $(n+2)$-angulated nor $n$-exact, see \cite{HLN,LZ,HZZ2}.
So a natural question is how to judge whether $n$-exangulated categories is $(n+2)$-angulated categories or $n$-exact categories.
Herschend, Liu and Nakaoka \cite{HLN} showed that $n$-exangulated categories which are $(n+2)$-angulated are those where the ${\rm Ext}^1$ bifunctor $\E$ is such that $\E(-,-) ={\rm Hom}(-,\Sigma-)$ for some auto-equivalence $\Sigma$ on the category, see \cite[Proposition 4.8]{HLN}.  They also characterized $n$-exact categories as $n$-exangulated categories for which inflations are monomorphisms and deflations are epimorphisms. Conversely,
$n$-exangulated categories which are $n$-exact categories are those where every
inflation is a monomorphism and every deflation is an epimorphism, see \cite[Proposition 4.37]{HLN}.

In this article, based on Msapato's idea, we offer a new characterization of $n$-exangulated categories which are $(n+2)$-angulated. More precisely, an $n$-exangulated category has the structure of an $(n+2)$-angulated category if and only if for any object $X$ in the category, the morphism $0\to X$ is a trivial deflation
and the morphism $X\to 0$ is a trivial inflation, see Theorem \ref{main}.
This result generalizes the work by Msapato \cite[Theorem 3.3]{M}.

This article is organized as follows. In Section 2, we review some elementary definitions
and facts on $n$-exangulated categories.  In Section 3, we recall the definition
of $(n+2)$-angulated categories and  prove our main result.

\section{Preliminaries}
Let us briefly recall some definitions and basic properties of $n$-exangulated categories from \cite{HLN}.
Throughout this article, let $\C$ be an additive category and $n$ be any positive integer.

\begin{definition}\cite[Definition 2.1]{HLN}
Suppose that $\C$ is equipped with an additive bifunctor $\E\colon\C\op\times\C\to\Ab$, where $\Ab$ is the category of abelian groups. For any pair of objects $A,C\in\C$, an element $\del\in\E(C,A)$ is called an {\it $\E$-extension} or simply an {\it extension}. We also write such $\del$ as ${}_A\del_C$ when we indicate $A$ and $C$.

Let ${}_A\del_C$ be any extension. Since $\E$ is a bifunctor, for any $a\in\C(A,A')$ and $c\in\C(C',C)$, we have extensions
$$ \E(C,a)(\del)\in\E(C,A')\ \ \text{and}\ \ \E(c,A)(\del)\in\E(C',A). $$
We simply denote them by $a_{\ast}\del$ and $c^{\ast}\del$.
In this terminology, we have
$$\E(c,a)(\del)=c^{\ast}a_{\ast}\del=a_{\ast}c^{\ast}\del\in\E(C',A').$$
For any $A,C\in\C$, the zero element ${}_A0_C=0\in\E(C,A)$ is called the {\it split $\E$-extension}.
\end{definition}

\begin{definition}\cite[Definition 2.3]{HLN}
Let ${}_A\del_C,{}_{A'}\del'_{C'}$ be any pair of $\E$-extensions. A {\it morphism} $(a,c)\colon\del\to\del'$ of extensions is a pair of morphisms $a\in\C(A,B)$ and $c\in\C(A',C')$ in $\C$, satisfying the equality
$$a_{\ast}\del=c^{\ast}\del'. $$
\end{definition}

\begin{definition}\cite[Definition 2.7]{HLN}
Let $\bf{C}_{\C}$ be the category of complexes in $\C$. As its full subcategory, define $\CC$ to be the category of complexes in $\C$ whose components are zero in the degrees outside of $\{0,1,\ldots,n+1\}$. Namely, an object in $\CC$ is a complex $X^{\mr}=\{X_i,d^X_i\}$ of the form
\[ X_0\xrightarrow{d^X_0}X_1\xrightarrow{d^X_1}\cdots\xrightarrow{d^X_{n-1}}X_n\xrightarrow{d^X_n}X_{n+1}. \]
We write a morphism $f^{\mr}\co X^{\mr}\to Y^{\mr}$ simply $f^{\mr}=(f^0,f^1,\ldots,f^{n+1})$, only indicating the terms of degrees $0,\ldots,n+1$.
\end{definition}

\begin{definition}\cite[Definition 2.11]{HLN}
By Yoneda lemma, any extension $\del\in\E(C,A)$ induces natural transformations
\[ \del\ssh\colon\C(-,C)\Rightarrow\E(-,A)\ \ \text{and}\ \ \del\ush\colon\C(A,-)\Rightarrow\E(C,-). \]
For any $X\in\C$, these $(\del\ssh)_X$ and $\del\ush_X$ are given as follows.
\begin{enumerate}
\item[\rm(1)] $(\del\ssh)_X\colon\C(X,C)\to\E(X,A)\ :\ f\mapsto f\uas\del$.
\item[\rm (2)] $\del\ush_X\colon\C(A,X)\to\E(C,X)\ :\ g\mapsto g\sas\delta$.
\end{enumerate}
We simply denote $(\del\ssh)_X(f)$ and $\del\ush_X(g)$ by $\del\ssh(f)$ and $\del\ush(g)$, respectively.
\end{definition}

\begin{definition}\cite[Definition 2.9]{HLN}
 Let $\C,\E,n$ be as before. Define a category $\AE:=\AE^{n+2}_{(\C,\E)}$ as follows.
\begin{enumerate}
\item[\rm(1)]  A pair $\Xd$ is an object of the category $\AE$ with $X^{\mr}\in\CC$
and $\del\in\E(X_{n+1},X_0)$, called an $\E$-attached
complex of length $n+2$, if it satisfies
$$(d_0^X)_{\ast}\del=0~~\textrm{and}~~(d^X_n)^{\ast}\del=0.$$
We also denote it by
$$X_0\xrightarrow{d_0^X}X_1\xrightarrow{d_1^X}\cdots\xrightarrow{d_{n-2}^X}X_{n-1}
\xrightarrow{d_{n-1}^X}X_n\xrightarrow{d_n^X}X_{n+1}\overset{\delta}{\dashrightarrow}.$$
\item[\rm (2)]  For such pairs $\Xd$ and $\langle Y^{\mr},\rho\rangle$,  $f^{\mr}\colon\Xd\to\langle Y^{\mr},\rho\rangle$ is
defined to be a morphism in $\AE$ if it satisfies $(f_0)_{\ast}\del=(f_{n+1})^{\ast}\rho$.

\end{enumerate}
\end{definition}

\begin{definition}\cite[Definition 2.13]{HLN}\label{def1}
 An {\it $n$-exangle} is an object $\Xd$ in $\AE$ that satisfies the listed conditions.
\begin{enumerate}
\item[\rm (1)] The following sequence of functors $\C\op\to\Ab$ is exact.
$$
\C(-,X_0)\xLongrightarrow{\C(-,\ d^X_0)}\cdots\xLongrightarrow{\C(-,\ d^X_n)}\C(-,X_{n+1})\xLongrightarrow{~\del\ssh~}\E(-,X_0)
$$
\item[\rm (2)] The following sequence of functors $\C\to\Ab$ is exact.
$$
\C(X_{n+1},-)\xLongrightarrow{\C(d^X_n,\ -)}\cdots\xLongrightarrow{\C(d^X_0,\ -)}\C(X_0,-)\xLongrightarrow{~\del\ush~}\E(X_{n+1},-)
$$
\end{enumerate}
In particular any $n$-exangle is an object in $\AE$.
A {\it morphism of $n$-exangles} simply means a morphism in $\AE$. Thus $n$-exangles form a full subcategory of $\AE$.
\end{definition}

\begin{definition}\cite[Definition 2.22]{HLN}
Let $\s$ be a correspondence which associates a homotopic equivalence class $\s(\del)=[{}_AX^{\mr}_C]$ to each extension $\del={}_A\del_C$. Such $\s$ is called a {\it realization} of $\E$ if it satisfies the following condition for any $\s(\del)=[X^{\mr}]$ and any $\s(\rho)=[Y^{\mr}]$.
\begin{itemize}
\item[{\rm (R0)}] For any morphism of extensions $(a,c)\co\del\to\rho$, there exists a morphism $f^{\mr}\in\CC(X^{\mr},Y^{\mr})$ of the form $f^{\mr}=(a,f_1,\ldots,f_n,c)$. Such $f^{\mr}$ is called a {\it lift} of $(a,c)$.
\end{itemize}
In such a case, we simple say that \lq\lq$X^{\mr}$ realizes $\del$" whenever they satisfy $\s(\del)=[X^{\mr}]$.

Moreover, a realization $\s$ of $\E$ is said to be {\it exact} if it satisfies the following conditions.
\begin{itemize}
\item[{\rm (R1)}] For any $\s(\del)=[X^{\mr}]$, the pair $\Xd$ is an $n$-exangle.
\item[{\rm (R2)}] For any $A\in\C$, the zero element ${}_A0_0=0\in\E(0,A)$ satisfies
\[ \s({}_A0_0)=[A\ov{\id_A}{\lra}A\to0\to\cdots\to0\to0]. \]
Dually, $\s({}_00_A)=[0\to0\to\cdots\to0\to A\ov{\id_A}{\lra}A]$ holds for any $A\in\C$.
\end{itemize}
Note that the above condition {\rm (R1)} does not depend on representatives of the class $[X^{\mr}]$.
\end{definition}

\begin{definition}\cite[Definition 2.23]{HLN}
Let $\s$ be an exact realization of $\E$.
\begin{enumerate}
\item[\rm (1)] An $n$-exangle $\Xd$ is called an $\s$-{\it distinguished} $n$-exangle if it satisfies $\s(\del)=[X^{\mr}]$. We often simply say {\it distinguished $n$-exangle} when $\s$ is clear from the context.
\item[\rm (2)]  An object $X^{\mr}\in\CC$ is called an {\it $\s$-conflation} or simply a {\it conflation} if it realizes some extension $\del\in\E(X_{n+1},X_0)$.
\item[\rm (3)]  A morphism $f$ in $\C$ is called an {\it $\s$-inflation} or simply an {\it inflation} if it admits some conflation $X^{\mr}\in\CC$ satisfying $d_X^0=f$.
\item[\rm (4)]  A morphism $g$ in $\C$ is called an {\it $\s$-deflation} or simply a {\it deflation} if it admits some conflation $X^{\mr}\in\CC$ satisfying $d_X^n=g$.
\end{enumerate}
\end{definition}

\begin{definition}\cite[Definition 2.27]{HLN}
For a morphism $f^{\mr}\in\CC(X^{\mr},Y^{\mr})$ satisfying $f^0=\id_A$ for some $A=X_0=Y_0$, its {\it mapping cone} $M_f^{\mr}\in\CC$ is defined to be the complex
\[ X_1\xrightarrow{d^{M_f}_0}X_2\oplus Y_1\xrightarrow{d^{M_f}_1}X_3\oplus Y_2\xrightarrow{d^{M_f}_2}\cdots\xrightarrow{d^{M_f}_{n-1}}X_{n+1}\oplus Y_n\xrightarrow{d^{M_f}_n}Y_{n+1} \]
where $d^{M_f}_0=\begin{bmatrix}-d^X_1\\ f_1\end{bmatrix},$
$d^{M_f}_i=\begin{bmatrix}-d^X_{i+1}&0\\ f_{i+1}&d^Y_i\end{bmatrix}\ (1\le i\le n-1),$
$d^{M_f}_n=\begin{bmatrix}f_{n+1}&d^Y_n\end{bmatrix}$.

{\it The mapping cocone} is defined dually, for morphisms $h^{\mr}$ in $\CC$ satisfying $h_{n+1}=\id$.
\end{definition}

\begin{definition}\cite[Definition 2.32]{HLN}
An {\it $n$-exangulated category} is a triplet $(\C,\E,\s)$ of additive category $\C$, additive bifunctor $\E\co\C\op\times\C\to\Ab$, and its exact realization $\s$, satisfying the following conditions.
\begin{itemize}
\item[{\rm (EA1)}] Let $A\ov{f}{\lra}B\ov{g}{\lra}C$ be any sequence of morphisms in $\C$. If both $f$ and $g$ are inflations, then so is $g\circ f$. Dually, if $f$ and $g$ are deflations, then so is $g\circ f$.

\item[{\rm (EA2)}] For $\rho\in\E(D,A)$ and $c\in\C(C,D)$, let ${}_A\langle X^{\mr},c\uas\rho\rangle_C$ and ${}_A\Yr_D$ be distinguished $n$-exangles. Then $(\id_A,c)$ has a {\it good lift} $f^{\mr}$, in the sense that its mapping cone gives a distinguished $n$-exangle $\langle M^{\mr}_f,(d^X_0)\sas\rho\rangle$.
 \item[{\rm (EA2$\op$)}] Dual of {\rm (EA2)}.
\end{itemize}
Note that the case $n=1$, a triplet $(\C,\E,\s)$ is a  $1$-exangulated category if and only if it is an extriangulated category, see \cite[Proposition 4.3]{HLN}.
\end{definition}

\begin{example}
From \cite[Proposition 4.34]{HLN} and \cite[Proposition 4.5]{HLN},  we know that $n$-exact categories and $(n+2)$-angulated categories are $n$-exangulated categories.
There are some other examples of $n$-exangulated categories
 which are neither $n$-exact nor $(n+2)$-angulated, see \cite{HLN,LZ,HZZ2}.
\end{example}

\begin{lemma}\emph{\cite[Lemma 2.12]{LZ}}\label{a1}
Let $(\C,\E,\s)$ be an $n$-exangulated category, and
$$A_0\xrightarrow{\alpha_0}A_1\xrightarrow{\alpha_1}A_2\xrightarrow{\alpha_2}\cdots\xrightarrow{\alpha_{n-2}}A_{n-1}
\xrightarrow{\alpha_{n-1}}A_n\xrightarrow{\alpha_n}A_{n+1}\overset{\delta}{\dashrightarrow}$$
be a distinguished $n$-exangle. Then we have the following exact sequences:
$$\C(-, A_0)\xrightarrow{}\C(-, A_1)\xrightarrow{}\cdots\xrightarrow{}
\C(-, A_{n+1})\xrightarrow{}\E(-, A_{0})\xrightarrow{}\E(-, A_{1})\xrightarrow{}\E(-, A_{2});$$
$$\C(A_{n+1},-)\xrightarrow{}\C(A_{n},-)\xrightarrow{}\cdots\xrightarrow{}
\C(A_0,-)\xrightarrow{}\E(A_{n+1},-)\xrightarrow{}\E(A_{n},-)\xrightarrow{}\E(A_{n-1},-).$$
\end{lemma}

\begin{lemma}\emph{\cite[Proposition 3.5]{HLN}}\label{a2}
Let ${}_A\langle X^{\mr},\delta\rangle_C$ and ${}_B\langle Y^{\mr},\rho\rangle_D$ be distinguished $n$-exangles. Suppose that we are given a commutative square
$$\xymatrix{
 X^0 \ar[r]^{{d_X^0}} \ar@{}[dr]|{\circlearrowright} \ar[d]_{a} & X^1 \ar[d]^{b}\\
 Y^0  \ar[r]_{d_Y^0} &Y^1
}
$$
in $\C$. Then there is a morphism $f^{\mr}\colon\Xd\to\langle Y^{\mr},\rho\rangle$ which satisfies $f^0=a$ and $f^1=b$.
\end{lemma}
\section{From $n$-exangulated categories to $(n+2)$-angulated categories}

For the convenience of the readers, we briefly recall the definition of an $(n+2)$-angulated category from \cite{GKO}. For more knowledge of $(n+2)$-angulated categories, see also \cite{BT}.

An $(n+2)$-$\Sigma^n$-$sequence$ in $\C$ is a sequence of objects and morphisms
$$A_0\xrightarrow{f_0}A_1\xrightarrow{f_1}A_2\xrightarrow{f_2}\cdots\xrightarrow{f_{n-1}}A_n\xrightarrow{f_n}A_{n+1}\xrightarrow{f_{n+1}}\Sigma^n A_0.$$
Its {\em left rotation} is the $(n+2)$-$\Sigma^n$-sequence
$$A_1\xrightarrow{f_1}A_2\xrightarrow{f_2}A_3\xrightarrow{f_3}\cdots\xrightarrow{f_{n}}A_{n+1}\xrightarrow{f_{n+1}}\Sigma^n A_0\xrightarrow{(-1)^{n}\Sigma^n f_0}\Sigma^n A_1.$$
A \emph{morphism} of $(n+2)$-$\Sigma^n$-sequences is  a sequence of morphisms $\varphi=(\varphi_0,\varphi_1,\cdots,\varphi_{n+1})$ such that the following diagram commutes
$$\xymatrix{
A_0 \ar[r]^{f_0}\ar[d]^{\varphi_0} & A_1 \ar[r]^{f_1}\ar[d]^{\varphi_1} & A_2 \ar[r]^{f_2}\ar[d]^{\varphi_2} & \cdots \ar[r]^{f_{n}}& A_{n+1} \ar[r]^{f_{n+1}}\ar[d]^{\varphi_{n+1}} & \Sigma^n A_0 \ar[d]^{\Sigma^n \varphi_0}\\
B_0 \ar[r]^{g_0} & B_1 \ar[r]^{g_1} & B_2 \ar[r]^{g_2} & \cdots \ar[r]^{g_{n}}& B_{n+1} \ar[r]^{g_{n+1}}& \Sigma^n B_0
}$$
where each row is an $(n+2)$-$\Sigma^n$-sequence. It is an {\em isomorphism} if $\varphi_0, \varphi_1, \varphi_2, \cdots, \varphi_{n+1}$ are all isomorphisms in $\C$.

\begin{definition}\cite[Definition 2.1]{GKO}
An $(n+2)$-\emph{angulated category} is a triple $(\C, \Sigma^n, \Theta)$, where $\C$ is an additive category, $\Sigma^n$ is an auto-equivalence of $\C$ ($\Sigma^n$ is called the $n$-suspension functor), and $\Theta$ is a class of $(n+2)$-$\Sigma^n$-sequences (whose elements are called $(n+2)$-angles), which satisfies the following axioms:
\begin{itemize}
\item[\textbf{(N1)}]
\begin{itemize}
\item[(a)] The class $\Theta$ is closed under isomorphisms, direct sums and direct summands.

\item[(b)] For each object $A\in\C$ the trivial sequence
$$ A\xrightarrow{1_A}A\rightarrow 0\rightarrow0\rightarrow\cdots\rightarrow 0\rightarrow \Sigma^nA$$
belongs to $\Theta$.

\item[(c)] Each morphism $f_0\colon A_0\rightarrow A_1$ in $\C$ can be extended to $(n+2)$-$\Sigma^n$-sequence: $$A_0\xrightarrow{f_0}A_1\xrightarrow{f_1}A_2\xrightarrow{f_2}\cdots\xrightarrow{f_{d-1}}A_n\xrightarrow{f_n}A_{n+1}\xrightarrow{f_{n+1}}\Sigma^n A_0.$$
\end{itemize}
\item[\textbf{(N2)}] An $(n+2)$-$\Sigma^n$-sequence belongs to $\Theta$ if and only if its left rotation belongs to $\Theta$.

\item[\textbf{(N3)}] Each solid commutative diagram
$$\xymatrix{
A_0 \ar[r]^{f_0}\ar[d]^{\varphi_0} & A_1 \ar[r]^{f_1}\ar[d]^{\varphi_1} & A_2 \ar[r]^{f_2}\ar@{-->}[d]^{\varphi_2} & \cdots \ar[r]^{f_{n}}& A_{n+1} \ar[r]^{f_{n+1}}\ar@{-->}[d]^{\varphi_{n+1}} & \Sigma^n A_0 \ar[d]^{\Sigma^n \varphi_0}\\
B_0 \ar[r]^{g_0} & B_1 \ar[r]^{g_1} & B_2 \ar[r]^{g_2} & \cdots \ar[r]^{g_{n}}& B_{n+1} \ar[r]^{g_{n+1}}& \Sigma^n B_0
}$$ with rows in $\Theta$, the dotted morphisms exist and give a morphism of  $(n+2)$-$\Sigma^n$-sequences.

\item[\textbf{(N4)}] In the situation of (N3), the morphisms $\varphi_2,\varphi_3,\cdots,\varphi_{n+1}$ can be chosen such that the mapping cone
$$A_1\oplus B_0\xrightarrow{\left(\begin{smallmatrix}
                                        -f_1&0\\
                                        \varphi_1&g_0
                                       \end{smallmatrix}
                                     \right)}
A_2\oplus B_1\xrightarrow{\left(\begin{smallmatrix}
                                        -f_2&0\\
                                        \varphi_2&g_1
                                       \end{smallmatrix}
                                     \right)}\cdots\xrightarrow{\left(\begin{smallmatrix}
                                        -f_{n+1}&0\\
                                        \varphi_{n+1}&g_n
                                       \end{smallmatrix}
                                     \right)} \Sigma^n A_0\oplus B_{n+1}\xrightarrow{\left(\begin{smallmatrix}
                                        -\Sigma^n f_0&0\\
                                        \Sigma^n\varphi_1&g_{n+1}
                                       \end{smallmatrix}
                                     \right)}\Sigma^nA_1\oplus\Sigma^n B_0$$
belongs to $\Theta$.
   \end{itemize}
From \cite{GKO}, we know that the classical triangulated categories are the special case $n=1$.

\end{definition}

\begin{definition}\label{b1} Let $(\mathscr{C},\mathbb{E},\mathfrak{s})$ be an $n$-exangulated category. Assume that $\mathscr{C}$ has the structure of an $(n+2)$-angulated category $(\C, \Sigma, \Theta)$. We say that this $(n+2)$-angulated structure is $\mathbb{E}$-\textit{compatible} if and only if for any $(n+2)$-angle
\[ X_0\xrightarrow{d^X_0}X_1\xrightarrow{d^X_1}\cdots\xrightarrow{d^X_{n-1}}X_n\xrightarrow{d^X_n}X_{n+1}\xrightarrow{d^X_{n+1}}\Sigma X_0,\]
we have that$$X_0\xrightarrow{d_0^X}X_1\xrightarrow{d_1^X}\cdots\xrightarrow{d_{n-2}^X}X_{n-1}
\xrightarrow{d_{n-1}^X}X_n\xrightarrow{d_n^X}X_{n+1}\overset{\delta}{\dashrightarrow}$$is a distinguished $n$-exangle for some $\del\in\E(X_{n+1},X_0)$.
\end{definition}

\begin{definition}\label{b1} Let $(\mathscr{C},\mathbb{E},\mathfrak{s})$ be an $n$-exangulated category and $X$ be any object in $\C$.
The morphism $X\to 0$ is called \emph{trivial inflation} if it can be embedded in a distinguished $n$-exangle
$$X\rightarrow0\rightarrow0\rightarrow\cdots\rightarrow0
\rightarrow Y\overset{}{\dashrightarrow}.$$
Dually the morphism $0\to X$ is called \emph{trivial deflation} if it can be embedded in a distinguished $n$-exangle
$$Z\rightarrow0\rightarrow0\rightarrow\cdots\rightarrow0
\rightarrow X\overset{}{\dashrightarrow}.$$
\end{definition}

\begin{remark}\label{remark}
In extriangulated categories, note that the inflation and the trivial inflation are the same, and the deflation and the trivial deflation are the same.
\end{remark}

 Our main result is the following.

\begin{theorem}\label{main}
 Let $(\mathscr{C}, \mathbb{E}, \mathfrak{s})$ be an $n$-exangulated category. Then $\mathscr{C}$ has an $\mathbb{E}$-compatible $(n+2)$-angulated structure $(\C, \Sigma, \Theta)$ if and only if for any object $ X \in \mathscr{C}$, the morphism $X \rightarrow 0$ is a trivial inflation and the morphism $0 \rightarrow X$ is a trivial deflation.
\end{theorem}

\textbf{We first give the proof of necessity.}

\begin{proof}
Let $(\mathscr{C}, \mathbb{E}, \mathfrak{s})$ be an $n$-exangulated category with an $\mathbb{E}$-compatible $(n+2)$-angulated structure $(\C, \Sigma, \Theta).$ For the morphism $1_X \colon X \rightarrow X$ in $\mathscr{C},$ by axioms (N1) and (N2), there exists an $(n+2)$-angle$$X\rightarrow 0\rightarrow\cdots\rightarrow 0\rightarrow \Sigma X\rightarrow\Sigma X$$ in $\Theta.$ Since the $(n+2)$-angulated structure is $\mathbb{E}$-compatible, we have that$$X\rightarrow 0\rightarrow\cdots\rightarrow 0\rightarrow \Sigma X\overset{\delta}{\dashrightarrow}$$ is a distinguished $n$-exangle for some $\delta \in \mathbb{E}(\Sigma X,X)$. So the morphism $X\rightarrow 0$ is a
trivial inflation.

Since $\Sigma$ is an auto-equivalence, there exists a functor $\Sigma'\colon \mathscr{C} \rightarrow \mathscr{C}$ such that $\Sigma' \circ \Sigma \cong 1_{\mathscr{C}}$ and $\Sigma \circ \Sigma' \cong 1_{\mathscr{C}}$. For any object $X\in\mathscr{C}$, $\Sigma$ is dense, then there exists some $Y\in \mathscr{C}$ such that $\Sigma Y \cong X$. By axioms (N1) and (N2), we have the following $(n+2)$-angle $$ Y\rightarrow 0\rightarrow\cdots\rightarrow 0\rightarrow \Sigma Y\xrightarrow{(-1)^{n}1_{\Sigma Y}}\Sigma Y,$$
which is isomorphic to the $(n+2)$-$\Sigma$-sequence$$ Y\rightarrow 0\rightarrow\cdots\rightarrow 0\rightarrow \Sigma Y\xrightarrow{1_{\Sigma Y}}\Sigma Y.$$
Then the above $(n+2)$-$\Sigma$-sequence is also an $(n+2)$-angle because $\Theta $ is closed under isomorphisms.
Since $\Sigma Y \cong X, Y \cong \Sigma'X$ and $X \cong \Sigma \Sigma'X,$ we have the following isomorphisms respectively, $\varphi_1  \colon \Sigma Y \rightarrow X, \varphi_2  \colon Y \rightarrow \Sigma'X$ and $\varphi_3 \colon X\rightarrow \Sigma \Sigma'X.$ Consider the following commutative diagram
$$\xymatrix{
Y \ar[r]\ar[d]^{\varphi_2} & 0 \ar[r]\ar@{=}[d] & \cdots  \ar[r] & 0 \ar[r]\ar@{=}[d]& \Sigma Y \ar[r]^{1_{\Sigma Y}}\ar[d]^{\varphi_{1}} & \Sigma Y \ar[d]^{ \varphi_3\varphi_1}\\
\Sigma'X \ar[r] & 0\ar[r] & \cdots  \ar[r] & 0 \ar[r]& X \ar[r]^{\varphi_3}& \Sigma \Sigma'X.
}$$
It is obvious that the above diagram is an isomorphism of  $(n+2)$-$\Sigma$-sequences. The first row is an $(n+2)$-angle, so is the second row. Since $(\C, \Sigma, \Theta)$ is $\mathbb{E}$-compatible, we have that $$\Sigma'X\rightarrow 0\rightarrow\cdots\rightarrow0
\rightarrow X\overset{\delta'}{\dashrightarrow}$$ is a distinguished $n$-exangle for some $\delta \in \mathbb{E}(X,\Sigma'X)$, that is, $0 \rightarrow X$ is a trivial deflation.
\end{proof}

Before proving the sufficiency of Theorem \ref{main}, we need some preparations. From now to the end of this section, assume that for any object $X\in\mathscr{C}$, the morphism $X \rightarrow 0$ is a trivial inflation, and the morphism $0 \rightarrow X$ is a trivial deflation. We will construct an auto-equivalence $\Sigma \colon \mathscr{C} \xrightarrow{~\sim~}\mathscr{C}$. Suppose that $X$ is any object in $\mathscr{C}$. Since the morphism $X \rightarrow 0$ is a trivial inflation, there exists a distinguished $n$-exangle $$X\rightarrow0\rightarrow0\rightarrow\cdots\rightarrow0
\rightarrow Z\overset{\delta_{X}}{\dashrightarrow}.$$
We write $Z:=\Sigma X$ and fix the distinguished $n$-exangle
$$X\rightarrow0\rightarrow0\rightarrow\cdots\rightarrow0
\rightarrow \Sigma X\overset{\delta_{X}}{\dashrightarrow}.$$

\begin{lemma}\label{lemma1}
For a morphism $f_0\colon {X_0}\rightarrow{Y_0}$, there exists the following commutative diagram of distinguished $n$-exangles
$$\xymatrix{
X_0\ar[r]\ar@{}[dr] \ar[d]^{f_0} &0 \ar[r] \ar@{}[dr]\ar@{=}[d] &0 \ar[r] \ar@{}[dr]\ar@{-->}[d]&\cdot\cdot\cdot \ar[r]\ar@{}[dr] &0 \ar[r] \ar@{}[dr]\ar@{-->}[d]&\Sigma X_0 \ar@{}[dr]\ar@{-->}[d]^{\Sigma f_0} \ar@{-->}[r]^-{\delta_{X_0}} &\\
{Y_0}\ar[r] &0\ar[r]&0 \ar[r] &\cdot\cdot\cdot \ar[r] &0\ar[r]  &\Sigma Y_0 \ar@{-->}[r]^-{\delta_{Y_0}} &,}
$$ where the morphism ${\Sigma f_0}$ is unique in $\C$.
\begin{proof}
 The leftmost square is commutative, then we have the required commutative diagram of distinguished $n$-exangles by Lemma \ref{a2}. The following sequences
$$\C(\Sigma X_0, Y_0)\xrightarrow{}\cdots\xrightarrow{}
\C(0, Y_0)\xrightarrow{}\C( X_0, Y_0)\xrightarrow{((\del_{X_0})\ush)_{Y_0}}\E(\Sigma X_0, Y_0)\xrightarrow{}\E(0, Y_0)\xrightarrow{}\E(0, Y_0)$$ and
$$\C(\Sigma X_0,Y_0)\xrightarrow{}\cdots\xrightarrow{}
\C(\Sigma X_0,0)\xrightarrow{}\C(\Sigma X_0,\Sigma Y_0)\xrightarrow{(({\del_{Y_0}})\ssh)_{\Sigma X_0}}\E(\Sigma X_0,Y_0)\xrightarrow{}\E(\Sigma X_0,0)$$
are exact by Lemma \ref{a1}.

Since $$\C(0,Y_0)=0=\E(0,Y_0),~~ \C(\Sigma X_0,0)=0=\E(\Sigma X_0,0)$$ in $\Ab$, we have that$$(({\del_{Y_0}})\ssh)_{\Sigma X_0}\colon\C(\Sigma X_0,\Sigma Y_0)\to\E(\Sigma X_0,Y_0)\ :\ g_0\mapsto g_0\uas\del_{Y_0}$$ and
$$((\del_{X_0})\ush)_{Y_0}\colon\C( X_0, Y_0)\to\E(\Sigma X_0,Y_0)\ :\ f_0\mapsto {f_0}\sas\delta _{X_0}$$
are group isomorphisms. Considering the commutative diagram as above, we can suppose there is a morphism $\Sigma f_0 \colon \Sigma X_0 \rightarrow \Sigma Y_0$ such that $${f_0}\sas\delta _{X_0} =((\del_{X_0})\ush)_{Y_0}(f_0 )=(\Sigma f_0 \uas)\delta_{Y_0}={(({\del_{Y_0}})\ssh)_{\Sigma X_0}}(\Sigma f_0).$$ Since $(({\del_{Y_0}})\ssh)_{\Sigma X_0}$ and $  ((\del_{X_0})\ush)_{Y_0}$ are isomorphisms, $\Sigma f_0$ is the unique morphism in $\C$.
\end{proof}
\end{lemma}
Moreover, by Lemma \ref{lemma1} we know that the $\Sigma 1_{X_0}$ is the unique morphism in $\mathscr{C}(\Sigma X_0, \Sigma X_0)$ such that $\delta _{X_0} = (\Sigma 1_{X_0})^{\ast}\delta _{X_0}$. Since  $(1_{\Sigma X_0})^{\ast}\delta _{X_0} = \delta _{X_0}$, we have  $\Sigma 1_{X_0} = 1_{\Sigma{X_0}}$. Suppose that $f_0 \colon X_0\rightarrow Y_0$ and $g_0 \colon Y_0 \rightarrow Z_0$ are two morphisms. Consider the following commutative diagram
$$\xymatrix{
X_0\ar[r]\ar@{}[dr] \ar[d]^{f_0} &0 \ar[r] \ar@{}[dr]\ar@{=}[d] &0 \ar[r] \ar@{}[dr]\ar@{-->}[d]&\cdot\cdot\cdot \ar[r]\ar@{}[dr] &0 \ar[r] \ar@{}[dr]\ar@{-->}[d]&\Sigma X_0 \ar@{}[dr]\ar@{-->}[d]^{\Sigma f_0} \ar@{-->}[r]^-{\delta_{X_0}} &\\Y_0\ar[r]\ar@{}[dr] \ar[d]^{g_0} &0 \ar[r] \ar@{}[dr]\ar@{=}[d] &0 \ar[r] \ar@{}[dr]\ar@{-->}[d]&\cdot\cdot\cdot \ar[r]\ar@{}[dr] &0 \ar[r] \ar@{}[dr]\ar@{-->}[d]&\Sigma Y_0 \ar@{}[dr]\ar@{-->}[d]^{\Sigma g_0} \ar@{-->}[r]^-{\delta_{Y_0}} &\\
{Z_0}\ar[r] &0\ar[r]&0 \ar[r] &\cdot\cdot\cdot \ar[r] &0\ar[r]  &\Sigma Z_0 \ar@{-->}[r]^-{\delta_{Z_0}} &.}
$$
We obtain that $(g_0 \circ f_0)_{\ast}\delta_{X_0} = (\Sigma g_0 \circ \Sigma f_0)^{\ast}\delta_{Z_0}.$ By Lemma \ref{lemma1}, we have that $\Sigma(g_0 \circ f_0)$ is the unique morphism in $\mathscr{C}(\Sigma X_0, \Sigma Z_0)$ such that  $(g_0 \circ f_0)_{\ast}\delta_{X_0} = (\Sigma(g_0 \circ f_0))^{\ast}\delta_{Z_0}$. This shows that $ \Sigma g_0 \circ \Sigma f_0 = \Sigma(g_0 \circ f_0)$.

\begin{proposition}
Let $(\mathscr{C}, \mathbb{E}, \mathfrak{s})$ be an $n$-exangulated category such that for any object $ X \in \mathscr{C}$, the morphism $X \rightarrow 0$ is a trivial inflation and the morphism $0 \rightarrow X$ is a trivial deflation. Suppose that $\Sigma \colon \mathscr{C} \rightarrow \mathscr{C}$ and $\Sigma^{\prime} \colon \mathscr{C} \rightarrow \mathscr{C}$ is two functors defined as above. Then $\Sigma$ and $\Sigma^{\prime}$ are naturally isomorphic.
\end{proposition}

\begin{proof}For any object $X_0\in\mathscr{C}$, we have the following commutative diagram
$$\xymatrix{
X_0\ar[r]\ar@{}[dr] \ar@{=}[d]&0 \ar[r] \ar@{}[dr]\ar@{=}[d] &0 \ar[r] \ar@{}[dr]\ar@{-->}[d]&\cdot\cdot\cdot \ar[r]\ar@{}[dr] &0 \ar[r] \ar@{}[dr]\ar@{-->}[d]&\Sigma X_0 \ar@{}[dr]\ar@{-->}[d]^{\eta_{X_0}} \ar@{-->}[r]^-{\delta_{X_0}} &\\
{X_0}\ar[r] &0\ar[r]&0 \ar[r] &\cdot\cdot\cdot \ar[r] &0\ar[r]  &\Sigma^{\prime} X_0 \ar@{-->}[r]^-{\delta^{\prime}_{X_0}} &
.}
$$
\ By Lemma \ref{lemma1}, we know that $\eta_{X_0}$ is an isomorphism in $\C$.

For any morphism $f_0:X_0\to Y_0$, we have the following two commutative diagrams
$$\xymatrix{
X_0\ar[r]\ar@{}[dr] \ar@{=}[d]&0 \ar[r] \ar@{}[dr]\ar@{=}[d] &0 \ar[r] \ar@{}[dr]\ar@{-->}[d]&\cdot\cdot\cdot \ar[r]\ar@{}[dr] &0 \ar[r] \ar@{}[dr]\ar@{-->}[d]&\Sigma X_0 \ar@{}[dr]\ar@{-->}[d]^{\eta_{X_0}} \ar@{-->}[r]^-{\delta_{X_0}} &\\X_0\ar[r]\ar@{}[dr] \ar[d]^{f_0}&0 \ar[r] \ar@{}[dr]\ar@{=}[d] &0 \ar[r] \ar@{}[dr]\ar@{-->}[d]&\cdot\cdot\cdot \ar[r]\ar@{}[dr] &0 \ar[r] \ar@{}[dr]\ar@{-->}[d]&\Sigma^{\prime} X_0 \ar@{}[dr]\ar@{-->}[d]^{\Sigma^{\prime}{f_0}} \ar@{-->}[r]^-{\delta^{\prime}_{X_0}} &\\
{Y_0}\ar[r] &0\ar[r]&0 \ar[r] &\cdot\cdot\cdot \ar[r] &0\ar[r]  &\Sigma^{\prime} Y_0 \ar@{-->}[r]^-{\delta^{\prime}_{Y_0}} &}
$$
and
$$\xymatrix{
X_0\ar[r]\ar@{}[dr]\ar@{}[dr] \ar[d]^{f_0} &0 \ar[r] \ar@{}[dr]\ar@{=}[d] &0 \ar[r] \ar@{}[dr]\ar@{-->}[d]&\cdot\cdot\cdot \ar[r]\ar@{}[dr] &0 \ar[r] \ar@{}[dr]\ar@{-->}[d]&\Sigma X_0 \ar@{}[dr]\ar@{-->}[d]^{\Sigma{f_0}} \ar@{-->}[r]^-{\delta_{X_0}} &\\Y_0\ar[r]\ar@{=}[d]&0 \ar[r] \ar@{}[dr]\ar@{=}[d] &0 \ar[r] \ar@{}[dr]\ar@{-->}[d]&\cdot\cdot\cdot \ar[r]\ar@{}[dr] &0 \ar[r] \ar@{}[dr]\ar@{-->}[d]&\Sigma Y_0 \ar@{}[dr]\ar@{-->}[d]^{\eta_{Y_0}} \ar@{-->}[r]^-{\delta_{Y_0}} &\\
{Y_0}\ar[r] &0\ar[r]&0 \ar[r] &\cdot\cdot\cdot \ar[r] &0\ar[r]  &\Sigma^{\prime} Y_0 \ar@{-->}[r]^-{\delta^{\prime}_{Y_0}} &.}
$$
By Lemma \ref{lemma1}, we have  $\Sigma^{\prime} f_0\circ\eta_{X_0}= \eta_{Y_0}\circ\Sigma f_0$, hence we have the following commutative diagram
$$\xymatrix{
 \Sigma X_0 \ar[r]^{\eta_{X_0}}_{\simeq} \ar@{}[dr]|{\circlearrowright} \ar[d]_{{\Sigma f_0}} &\Sigma^{\prime} X_0 \ar[d]^{\Sigma^{\prime} f_0}\\
\Sigma Y_0  \ar[r]_{\eta_{Y_0}}^{\simeq} &\Sigma^{\prime} Y_0.
}
$$
So $\eta$ defines a natural isomorphism, $\Sigma$ and $\Sigma^{\prime}$ are naturally isomorphic.
\end{proof}

\begin{remark}
Through the above discussion, we know that $\Sigma \colon \mathscr{C} \rightarrow \mathscr{C}$ is a well-defined functor.
\end{remark}

\begin{proposition}\label{proposition4} Let $(\mathscr{C}, \mathbb{E}, \mathfrak{s})$ be an $n$-exangulated category such that for any object $ X \in \mathscr{C}$, the morphism $X \rightarrow 0$ is a trivial inflation and the morphism $0 \rightarrow X$ is a trivial deflation. Assume that $\Sigma \colon \mathscr{C} \rightarrow \mathscr{C}$ as defined above. Then $\Sigma \colon \mathscr{C} \rightarrow \mathscr{C}$ is an auto-equivalence.
\end{proposition}

\begin{proof}
 We first show that $\Sigma$ is an additive functor. Since the following diagram commutes
$$\xymatrix{
0\ar[r]\ar@{}[dr]\ar@{=}[d]&0 \ar[r] \ar@{}[dr]\ar@{=}[d] &0 \ar[r] \ar@{}[dr]\ar@{-->}[d]&\cdot\cdot\cdot \ar[r]\ar@{}[dr] &0 \ar[r] \ar@{}[dr]\ar@{-->}[d]&0 \ar@{}[dr]\ar@{-->}[d] \ar@{-->}[r]^-{0} &\\
0\ar[r] &0\ar[r]&0 \ar[r] &\cdot\cdot\cdot \ar[r] &0\ar[r]  &\Sigma 0 \ar@{-->}[r]^-{\delta_{0}} &
,}
$$
we have $\Sigma 0 \cong 0$.
Let $X_0,Y_0$ be any pair of objects and consider the following distinguished $n$-exangle
$$\xymatrix{X_0\oplus Y_0 \ar[r]^{} &0 \ar[r]^{} \ar[r] &\cdot\cdot\cdot \ar[r]^{} &0 \ar[r]^{} \ar[r] &\Sigma (X_0\oplus Y_0) \ar@{-->}[r]^-{\delta_{X_0\oplus Y_0}} &.}$$
By \cite[Remark 2.14]{HLN}, we have that
$$\xymatrix{X_0\oplus Y_0 \ar[r]^{} &0 \ar[r]^{} \ar[r] &\cdot\cdot\cdot \ar[r]^{} &0 \ar[r]^{} \ar[r] &\Sigma X_0\oplus \Sigma Y_0 \ar@{-->}[r]^-{\delta_{X_0}\oplus\delta_{ Y_0}} &}$$
is a distinguished $n$-exangle.  Note that the following diagram
$$\xymatrix{
X_0\oplus Y_0\ar[r]\ar@{}[dr]\ar@{=}[d]&0 \ar[r] \ar@{}[dr]\ar@{=}[d] &0 \ar[r] \ar@{}[dr]\ar@{-->}[d]&\cdot\cdot\cdot \ar[r]\ar@{}[dr] &0 \ar[r] \ar@{}[dr]\ar@{-->}[d]&\Sigma (X_0\oplus Y_0)\ar@{}[dr]\ar@{-->}^{\eta}[d] \ar@{-->}[r]^-{{\delta_{X_0\oplus Y_0}}} &\\
X_0\oplus Y_0\ar[r] &0\ar[r]&0 \ar[r] &\cdot\cdot\cdot \ar[r] &0\ar[r]  &\Sigma X_0\oplus \Sigma Y_0\ar@{-->}[r]^-{{\delta_{X_0}\oplus \delta_{Y_0}}} &
}
$$
commutes. By Lemma \ref{lemma1}, $\eta$ is an isomorphism, so $\Sigma (X _0\oplus Y_0) \cong \Sigma X_0 \oplus \Sigma Y_0$. This shows that $\Sigma$ is an additive functor.

In order to show that $\Sigma$ is an equivalence, we need to show that $\Sigma$ is dense,
full and faithful.

We first prove that it is dense. Assume that $Y_0$ is any object in $\mathscr{C}$.
By hypothesis, the morphism $0 \rightarrow Y_0$ is a trivial deflation, so there exists a distinguished $n$-exangle
$$X_0\rightarrow0\rightarrow0\rightarrow\cdots\rightarrow0
\rightarrow Y_0\overset{\delta}{\dashrightarrow}.$$
Thus we obtain the following commutative diagram
$$\xymatrix{
X_0\ar[r]\ar@{}[dr]\ar@{=}[d]&0 \ar[r] \ar@{}[dr]\ar@{=}[d] &0 \ar[r] \ar@{}[dr]\ar@{-->}[d]&\cdot\cdot\cdot \ar[r]\ar@{}[dr] &0 \ar[r] \ar@{}[dr]\ar@{-->}[d]&\Sigma X_0\ar@{}[dr]\ar@{-->}[d]^{\eta} \ar@{-->}[r]^-{{\delta_{X_0}}} &\\
X_0\ar[r] &0\ar[r]&0 \ar[r] &\cdot\cdot\cdot \ar[r] &0\ar[r]  &Y_0 \ar@{-->}[r]^-{\delta} &.
}
$$
By Lemma \ref{lemma1}, the morphism $\eta \colon \Sigma X_0 \rightarrow Y_0$ is an isomorphism, so $\Sigma$ is dense.

Let $X_0,Y_0$ be a pair of objects in $\mathscr{C}$. Consider the map $\Sigma_{X_0,Y_0} \colon \mathscr{C}(X_0,Y_0) \rightarrow \mathscr{C}(\Sigma X_0, \Sigma Y_0),$ where $\Sigma_{X_0,Y_0}(f_0) = \Sigma f_0$. Assume $\Sigma f_0 = \Sigma f^{\prime}_0$,
by Lemma \ref{lemma1} we obtain that $${f_0}_{\ast}\delta_{X_0} = (\Sigma f_0)^{\ast}\delta_{Y_0} = (\Sigma {f^{\prime}_0})^{\ast}\delta_{Y_0} = {f^{\prime}_0}_{*}\delta_{X_0}.$$
Hence
 $$((\del_{X_0})\ush)_{Y_0}(f_0) = ((\del_{X_0})\ush)_{Y_0}({f^{\prime}_0}).$$ Since $((\del_{X_0})\ush)_{Y_0}$ is an isomorphism, we get that $f_0 = {f^{\prime}_0}$, that is, $\Sigma_{X_0,Y_0}$ is injective.

Let $g_0 \colon \Sigma X_0 \rightarrow \Sigma Y_0$ be any morphism in $\mathscr{C}(\Sigma X_0, \Sigma Y_0)$. We have the following commutative diagram
$$\xymatrix{
X_0\ar[r]\ar@{}[dr]\ar@{-->}[d]^{f_0}&0 \ar[r] \ar@{}[dr]\ar@{-->}[d] &0 \ar[r] \ar@{}[dr]\ar@{-->}[d]&\cdot\cdot\cdot \ar[r]\ar@{}[dr] &0 \ar[r] \ar@{}[dr]\ar[d]&\Sigma X_0\ar@{}[dr]\ar[d]^{g_0} \ar@{-->}[r]^-{{\delta_{X_0}}} &\\Y_0\ar[r] &0\ar[r]&0 \ar[r] &\cdot\cdot\cdot \ar[r] &0\ar[r]  &\Sigma Y_0 \ar@{-->}[r]^-{\delta_{Y_0}} &
}
$$
of distinguished $n$-exangles.
By the dual of Lemma \ref{a2}, there exists a morphism $f_0 \colon X_0 \rightarrow Y_0$ such that  ${f_0}_{\ast}\delta_{X_0} = {g^{\ast}_0}\delta_{Y_0}$. Since $\Sigma f_0$ is the unique morphism in $\mathscr{C}(\Sigma X_0, \Sigma Y_0)$ such that ${f_0}_{\ast}\delta_{X_0} = (\Sigma f_0)^{\ast}\delta_{Y_0}$, we have that $g_0 = \Sigma f_0$. This shows that the map $\Sigma_{X_0,Y_0}$ is surjective. Thus we conclude that $\Sigma$ is full and faithful. This completes the proof.
\end{proof}

Now let $\mathbf{E}^{1} \colon \mathscr{C}^{\text{op}} \times \mathscr{C} \rightarrow Ab$ be the bifunctor defined by $\mathbf{E}^{1}(-,-) := \mathscr{C}(-,\Sigma -).$ We will show that it is an additive bifunctor.

\begin{lemma}\label{lemma6} Let $(\mathscr{C}, \mathbb{E}, \mathfrak{s})$ be an $n$-exangulated category such that for any object $ X \in \mathscr{C}$, the morphism $X \rightarrow 0$ is a trivial inflation and the morphism $0 \rightarrow X$ is a trivial deflation. Assume that $X_0,Y_0$ is a pair of objects in $\mathscr{C}$. Then $\mathbf{E}^{1}(X_0,Y_0):=\mathscr{C}(X_0,\Sigma Y_0) \cong \mathbb{E}(X_0,Y_0)$.
\begin{proof}
For the distinguished $n$-exangle
$$Y_0\rightarrow0\rightarrow 0\rightarrow\cdots\rightarrow 0
\rightarrow\Sigma Y_0\overset{ \delta_{Y_0}}{\dashrightarrow},$$
the following sequence
$$\C(X_0,Y_0)\xrightarrow{}\C(X_0,0)\xrightarrow{}\cdots\xrightarrow{}
\C(X_0,0)\xrightarrow{}\C(X_0,\Sigma Y_0)\xrightarrow{({\del_{Y_0}}\ssh)_{X_0}}\E(X_0,Y_0)\xrightarrow{}\E(X_0,0)$$
 is exact by Lemma \ref{a1}. Since $\mathscr{C}(X_0,0)=0=\mathbb{E}(X_0,0),$ we obtain that $$({\del_{Y_0}}\ssh)_{X_0}\colon \mathscr{C}(X_0,\Sigma Y_0) \rightarrow \mathbb{E}(X_0,Y_0)$$ is an isomorphism.
\end{proof}

\end{lemma}
\begin{lemma}\label{lemma3} The functor $\mathbf{E}^{1} \colon \mathscr{C}^{\rm op} \times \mathscr{C} \rightarrow Ab$ is an additive bifunctor.
\end{lemma}

\begin{proof}
By Lemma \ref{lemma6}, we get that $\mathbf{E}^{1}(X_0,0) \cong \mathbb{E}(X_0,0) \cong {0} \cong \mathbb{E}(0,X_0) \cong \mathbf{E}^{1}(0,X_0)$ and $\mathbf{E}^{1}(X_0,Y_0 \oplus Z_0) \cong \mathbb{E}(X_0,Y_0 \oplus Z_0) \cong \mathbb{E}(X_0,Y_0) \oplus \mathbb{E}(X_0,Z_0) \cong \mathbf{E}^{1}(X_0,Y_0) \oplus \mathbf{E}^{1}(X_0,Z_0)$. This shows that $\mathbf{E}^{1}$ is additive in the second argument. By dually, it is also additive in the first argument. Hence $\mathbf{E}^{1}$ is an additive bifunctor.
\end{proof}

\begin{lemma} Let $(\mathscr{C}, \mathbb{E}, \mathfrak{s})$ be an $n$-exangulated category such that for every object $ X \in \mathscr{C}$, the morphism $X \rightarrow 0$ is a trivial inflation and the morphism $0 \rightarrow X$ is a trivial deflation. Assume that $Z$ is an object in $\mathscr{C}$ and $f_0 \colon X_0 \rightarrow Y_0$ is a morphism in $\mathscr{C}$. Then we obtain that $\mathbf{E}^{1}(Z,f_0) \cong \mathbb{E}(Z,f_0)$ and $\mathbf{E}^{1}({f_0^{\rm op}},Z) \cong \mathbb{E}({f_0^{\rm op}},Z)$ in the category of morphisms of $\Ab$.
\end{lemma}

\begin{proof}
By Lemma \ref{lemma6}, we obtain that the map $(({\delta_{{X_0}})\ssh})_{Z}: \mathbf{E}^{1}(Z,X_0) \rightarrow \mathbb{E}(Z,X_0)$, where $(({\delta_{{X_0}})\ssh})_{Z}(\varepsilon) = \varepsilon^{*}\delta_{X_0},$ and the map $(({\delta_{{Y_0}})\ssh})_{Z} : \mathbf{E}^{1}(Z,Y_0) \rightarrow \mathbb{E}(Z,Y_0),$ where $(({\delta_{{Y_0}})\ssh})_{Z}(\varepsilon) = \varepsilon^{\ast}\delta_{Y_0}$, are group isomorphisms. Moreover, it is not hard to check that the following diagram
$$\xymatrix@C=1.4cm{
 \mathbf{E}^{1}{(Z,X_0)} \ar[r]^{{\mathbf{E}^{1}(Z,f_0)}} \ar{}\ar[d]_{{((\delta_{X_0}){\ssh})_{Z}}} & \mathbf{E}^{1}(Z,Y_0)\ar[d]^{{((\delta_{Y_0}){\ssh})_{Z}}}\\
\mathbf{E}(Z,X_0)  \ar[r]_{{\mathbf{E}(Z,f_0)}} &\mathbf{E}(Z,Y_0)
}
$$
commutes for any $X_0,Y_0$ in $\C$. So $\mathbf{E}^{1}(Z,f_0) \cong \mathbb{E}(Z,f_0)$ in the category of morphisms of $\Ab$. Dually we prove the remaining statement.
\end{proof}

Now we define a correspondence $\mathfrak{r}$ for the category $\mathscr{C}$ equipped with the additive bifunctor $\mathbf{E}^{1}$, which will associate an equivalence class $\mathfrak{r}(\varepsilon)$ to any extension $\varepsilon \in \mathbf{E}^{1}(C,A)$.
Assume that $A,C$ is a pair of objects in $\mathscr{C}$, By Lemma \ref{lemma6}, we have $\mathbf{E}^{1}(C,A) \cong \mathbb{E}(C,A)$. Hence any $\varepsilon$ corresponds to $\varepsilon^{\ast}\delta_{A} \in \mathbb{E}(C,A)$. Therefore, we denote $\mathfrak{r}(\varepsilon):= \mathfrak{s}(\varepsilon^{\ast}\delta_{A})$.

\begin{proposition}\label{proposition2} Let $(\mathscr{C}, \mathbb{E}, \mathfrak{s})$ be an $n$-exangulated category such that for any object $ X \in \mathscr{C}$, the morphism $X\rightarrow 0$ is a trivial inflation and the morphism $0 \rightarrow X$ is a trivial deflation. Assume that $\mathfrak{r}$ is the correspondence which associates the equivalence class $\mathfrak{r}(\varepsilon) = \mathfrak{s}(\varepsilon^{\ast}\delta_{X_0})$ to any $\mathbf{E}^{1}$-extension $\varepsilon \in \mathbf{E}^{1}(Y_0,X_0)$, for any pair of objects $X_0,Y_0$ in $\mathscr{C}$. Then $\mathfrak{r}$ is an exact realization of $\mathbf{E}^{1}$.
\end{proposition}

\begin{proof}
Let $\varepsilon \in \mathbf{E}^{1}(Y_0,X_0)$ and $\varepsilon^{\prime} \in \mathbf{E}^{1}({Y^{\prime}_0},{X^{\prime}_0})$ be two $\mathbf{E}^{1}$-extensions with $$\mathfrak{r}(\varepsilon) = [X_0 \overset{}{\longrightarrow} X_1{\longrightarrow}\cdots {\longrightarrow} X_n \overset{}{\longrightarrow}Y_0]~\mbox{ and}~ \mathfrak{r}(\varepsilon^{\prime}) = [{X^{\prime}_0}\overset{}{\longrightarrow} {X^{\prime}_1}{\longrightarrow}\cdots {\longrightarrow} {X^{\prime}_n} \overset{}{\longrightarrow}{Y^{\prime}_0}].$$
Assume that we have a morphism of $\mathbf{E}^{1}$-extensions $(f_0,g_0) \colon \varepsilon \rightarrow \varepsilon^{\prime}$, namely, $$\mathbf{E}^{1}(Y_0,f_0)(\varepsilon)=\Sigma f_0\circ \varepsilon = \mathbf{E}^{1}(g_0,{X^{\prime}_0})(\varepsilon^{\prime})= \varepsilon^{\prime} \circ g_0.$$ Since $\mathfrak{r}(\varepsilon) = \mathfrak{s}(\varepsilon^{\ast}\delta_{X_0})$ and $\mathfrak{r}(\varepsilon^{\prime}) = \mathfrak{s}((\varepsilon^{\prime})^{\ast}\delta_{{X^{\prime}_0}})$, we have the following diagram
$$\xymatrix{
X_0\ar[r]\ar@{}[dr] \ar[d]^{f_0} &X_1 \ar[r] \ar@{}[dr]\ar@{-->}[d]^{f_1} &X_2 \ar[r] \ar@{}[dr]\ar@{-->}[d]^{f_2}&\cdot\cdot\cdot \ar[r]\ar@{}[dr] &X_n \ar[r] \ar@{}[dr]\ar@{-->}[d]^{f_n}&Y_0 \ar@{}[dr]\ar[d]^{g_0} \ar@{-->}[r]^-{\varepsilon^{\ast}\delta_{X_0}} &\\
{X^{\prime}_0}\ar[r] &{X^{\prime}_1}\ar[r]&{X^{\prime}_2} \ar[r] &\cdot\cdot\cdot \ar[r] &{X^{\prime} _n}\ar[r]  &{Y^{\prime}_0} \ar@{-->}[r]^-{(\varepsilon^{\prime})^{\ast}\delta_{{X^{\prime}_0}}} &}
$$
of distinguished $n$-exangles.
Since ${f_0}_{\ast}\delta_{X_0} = (\Sigma f_0)^{\ast}\delta_{{X^{\prime}_0}},$
we obtain
$${f_0}_{\ast}\varepsilon^{\ast}\delta_{X_0} = \varepsilon^{\ast}{f_0}_{\ast}\delta_{X_0} = \varepsilon^{\ast}(\Sigma f_0)^{\ast}\delta_{{X^{\prime}_0}} = (\Sigma f_0\circ \varepsilon)^{\ast}\delta_{{X^{\prime}_0}} = (\varepsilon^{\prime} \circ g_0)^{\ast}\delta_{{X^{\prime}_0}}= {g_0}^{\ast} (\varepsilon^{\prime})^{\ast}\delta_{{X^{\prime}_0}},$$
that is, $(f_0,g_0) \colon \varepsilon^{\ast}\delta_{X_0} \rightarrow (\varepsilon^{\prime})^{\ast}\delta_{{X^{\prime}_0}}$ is a morphism of $\mathbb{E}$-extensions. Because $\mathfrak{s}$ is a realization, there is a morphism $(f_0,f_1,f_2,\cdots ,f_{n-1},f_n,g_0)$ such that the above diagram commutes. This proves that $\mathfrak{r}$ is a realization.

It remains to show that $\mathfrak{r}$ is an exact realization.

\item[{(R1)}]For any $\mathfrak{r}(\varepsilon) = [X_0 \overset{f_0}{\longrightarrow} X_1\overset{f_1}{\longrightarrow}\cdots {\longrightarrow} X_n \overset{f_n}{\longrightarrow}X_{n+1}]=[X^{\mr}]$, consider the following diagram
$$\xymatrix@C=1.2cm{
\C(-,X_0)\ar[r]\ar@{=}[d] &\C(-,X_1)\ar[r]\ar@{=}[d] \ar[r] & \cdots \ar[r] & \C(-,X_{n+1}) \ar[r]^{\varepsilon_{\ssh}}\ar@{=}[d]& {\mathbf{E}}^{1}(-,X_0)\ar[d]^{(\delta_{X_0}){\ssh}}\\
\C(-,X_0)\ar[r] &\C(-,X_1)\ar[r]  \ar[r] & \cdots \ar[r]& \C(-,X_{n+1})\ar[r]^{(\varepsilon^{\ast}\delta_{X_0})_{\ssh}}& {\mathbb{E}}(-,X_0).
}$$
For every $k\in\C(-,X_{n+1})$, we have
$$((({\delta_{{X_0}})\ssh})(\varepsilon_{\ssh}(k))= (({\delta_{{X_0}})\ssh})(k^{\ast}\varepsilon)=(({\delta_{{X_0}})\ssh})(\varepsilon k)=(\varepsilon k)^{\ast}\delta_{X_0}= ({k}^{\ast}{\varepsilon}^{\ast})\delta_{X_0}={k}^{\ast}({\varepsilon}^{\ast}\delta_{X_0})=(\varepsilon^{*}\delta_{X_0})_{\ssh}(k).$$
Hence the above diagram is commutative. Since the two sequences of abelian groups are isomorphic and the bottom row is exact, the top row is exact, that is,
$${\C(-,X_0)}\xLongrightarrow{\C(-,\ f_0)}\cdots\xLongrightarrow{\C(-,\ f_n)}\C(-,X_{n+1})\xLongrightarrow{~\varepsilon\ssh~}\E(-,X_0)$$ is exact. By the similar argument as above,
$${\C(X_{n+1},-)}\xLongrightarrow{\C(f_{n+1},\ -)}\cdots\xLongrightarrow{\C(f_0,\ -)}\C(X_0,-)\xLongrightarrow{~\varepsilon^{\ssh}~}\E(X_{n+1},-)$$ is also exact, therefore the piar $(X^{\mr},\varepsilon)$ is an $n$-exangle.
\item[{(R2)}]For any $X_0\in\C$, the zero element${}_{X_0}0_0=0\in\E^{1}(0,X_0)$, then $$\mathfrak{r}({}_{X_0}0_0)=\mathfrak{s}(0^{*}\delta_{X_0}) = \mathfrak{s}(0)=\mathfrak{s}({}_{X_0}0_0)=[X_0 \overset{1_{X_0}}{\longrightarrow} X_0{\longrightarrow} 0{\longrightarrow}\cdots {\longrightarrow} 0 {\longrightarrow}0].$$
Similarly, we obtain
$$\mathfrak{r}({}_00_{X_0})=[0 \overset{}{\longrightarrow} 0{\longrightarrow}\cdots{\longrightarrow} 0 {\longrightarrow} X_0\overset{1_{X_0}}{\longrightarrow}X_0].$$
This shows that $\mathfrak{r}$ is an exact realization.
\end{proof}

\begin{proposition}\label{proposition1}The triple $(\mathscr{C},\mathbf{E}^{1},\mathfrak{r})$ satisfies the axioms {\rm (EA1)}, {\rm (EA2)} and  $\text{\rm (EA2)}^{\rm op}$.
\end{proposition}

\begin{proof}{\rm (EA1)} Let $A\ov{f}{\lra}B\ov{g}{\lra}C$ be any sequence of morphisms in ${\C}$.
Assume that $f$ and $g$ are $\mathfrak{r}$-inflations, then there are two
$\mathfrak{r}$-conflations $X^{\mr}\in\CC$ and $Y^{\mr}\in\CC$ satisfying $f={f_0}$ and $g={g_0}$, respectively.
Thus we assume
$$\mathfrak{r}(\varepsilon)=[A\overset{f}{\longrightarrow}B\overset{f_1}{\longrightarrow}X_1\overset{f_2}{\longrightarrow}\cdots\overset{f_{n-1}}{\longrightarrow} X_{n-1}\overset{f_n}{\longrightarrow} X_n]$$
$$\mathfrak{r}(\varepsilon^{\prime})=[B\overset{g}{\longrightarrow}C\overset{g_1}{\longrightarrow}Y_1\overset{g_2}{\longrightarrow}\cdots\overset{g_{n-1}}{\longrightarrow} Y_{n-1}\overset{g_n}{\longrightarrow} Y_n].$$
Since $\mathfrak{r}(\varepsilon) = \mathfrak{s}(\varepsilon^{\ast}\delta_{A})$ and $\mathfrak{r}(\varepsilon^{\prime}) = \mathfrak{s}((\varepsilon^{\prime})^{\ast}\delta_{B})$, we have
$$\mathfrak{s}(\varepsilon^{\ast}\delta_{A})=[A\overset{f}{\longrightarrow}B\overset{f_1}{\longrightarrow}X_1\overset{f_2}{\longrightarrow}\cdots\overset{f_{n-1}}{\longrightarrow} X_{n-1}\overset{f_n}{\longrightarrow} X_n]$$
$$\mathfrak{s}((\varepsilon^{\prime})^{\ast}\delta_{B})=[B\overset{g}{\longrightarrow}C\overset{g_1}{\longrightarrow}Y_1\overset{g_2}{\longrightarrow}\cdots\overset{g_{n-1}}{\longrightarrow} Y_{n-1}\overset{g_n}{\longrightarrow} Y_n].$$
This show $f$ and $g$ are $\mathfrak{s}$-inflations. By (EA1), there exists an $\mathfrak{s}$-conflation $Z^{\mr}\in\CC$ satisfying
$$\mathfrak{s}(\delta)=[A\overset{gf}{\longrightarrow}C\overset{h_0}{\longrightarrow}Z_1\overset{h_1}{\longrightarrow}\cdots\overset{h_{n-1}}{\longrightarrow} Z_{n-1}\overset{h_n}{\longrightarrow} Z_n].$$
Since $\mathbf{E}^{1}(Z_n,A) \cong\mathbb{E} (Z_n,A)$, we get that there exists $\varepsilon^{\prime \prime} \in \mathscr{C}(Z_n,\Sigma A)$ such that $\delta = (\varepsilon^{\prime \prime})^{*}\delta_{A}$. Moreover, since $\mathfrak{r}(\varepsilon^{\prime \prime}) = \mathfrak{s}((\varepsilon^{\prime \prime})^{\ast}\delta_{A})$, we obtain $$\mathfrak{r}(\varepsilon^{\prime \prime})=[A\overset{gf}{\longrightarrow}C\overset{h_0}{\longrightarrow}Z_1\overset{h_1}{\longrightarrow}\cdots\overset{h_{n-1}}{\longrightarrow} Z_{n-1}\overset{h_n}{\longrightarrow} Z_n].$$
So $g\circ f$ are $\mathfrak{r}$-inflations. Dually, we can show if $f$ and $g$ are $\mathfrak{r}$-deflations, then so is $g\circ f$.

{\rm (EA2)} For any $\delta\in\mathbf{E}^{1}(D,A)$ and $c\in\C(C,D)$,
let ${}_A\langle X^{\mr},c\uas\delta\rangle_C$ and ${}_A\langle Y^{\mr}, \delta\rangle_D$ be distinguished $n$-exangles. We may assume
$$\mathfrak{r}(\delta)=[A\overset{g_0}{\longrightarrow}Y_1\overset{g_1}{\longrightarrow}Y_1\overset{g_2}{\longrightarrow}\cdots\overset{g_{n-1}}{\longrightarrow} Y_{n}\overset{g_n}{\longrightarrow} D],$$
$$\mathfrak{r}(c^{\ast}\delta)=[A\overset{f_0}{\longrightarrow}X_1\overset{f_1}{\longrightarrow}X_2\overset{f_2}{\longrightarrow}\cdots\overset{f_{n-1}}{\longrightarrow} X_{n}\overset{f_n}{\longrightarrow}C].$$
Since $\mathfrak{r}(\delta) = \mathfrak{s}(\delta^{\ast}\delta_{A})$ and $\mathfrak{r}(c^{\ast}\delta) = \mathfrak{s}({(c^{\ast}\delta)^{\ast}}\delta_{A})$, $(1_A, c)$ has a good lift of $f^{\centerdot}$ in $(\C,\E,\mathfrak{s})$
$$\xymatrix{
A \ar[r]^{f_0} \ar@{}[dr]|{\circlearrowright} \ar@{=}[d] &X_1 \ar[r]^{f_1} \ar@{}[dr]|{\circlearrowright} \ar[d] &\cdot\cdot\cdot \ar[r]^{f_{n-1}} \ar@{}[dr]|{\circlearrowright} &X_n \ar[d] \ar@{}[dr]|{\circlearrowright} \ar[r]^{f_n} &C \ar[d]^{c} \ar@{-->}[r]^-{{(c^{\ast}\delta)^{\ast}}\delta_{A}} &\\
A \ar[r]_{g_0} &Y_1 \ar[r]_{g_1} &\cdot\cdot\cdot \ar[r]_{g_{n-1}} &Y_n \ar[r]_{g_n}  &D \ar@{-->}[r]^-{\delta^{\ast}\delta_{A}} &.}
$$
Moreover, the mapping cone gives a distinguished $n$-exangle ${}\langle M^{\mr}_f,(f_0)\sas\delta^{\ast}\delta_{A}\rangle$ in $(\C,\E,\s)$. Since $\mathbf{E}^{1}(D,X_1) \cong\mathbb{E}(D,X_1)$, there exists $\varepsilon\in \mathscr{C}(D,\Sigma X_1)$ such that $\varepsilon^{\ast}\delta_{X_1} = (f_0)_{\ast}\delta^{\ast}\delta_A$. Therefore $\varepsilon^{\ast}\delta_{X_1} = (f_0)_{\ast}\delta^{\ast}\delta_A= \delta^{\ast}(\Sigma f_0)^{\ast}\delta_{X_1}= (\Sigma f_0\circ\delta)^{\ast}\delta_{X_1}$. Because $$(({\delta_{X_1}})\ssh)(\varepsilon)=\varepsilon^{\ast}\delta_{X_1}=(\Sigma f\circ\delta)^{\ast}\delta_{X_1}=(({\delta_{X_1}})\ssh)(\Sigma f_0\circ\delta)$$ and $(({\delta_{X_1}})\ssh)$ is an isomorphism, we obtain $\varepsilon=\Sigma f_0\circ\delta$. Note that $\Sigma f_0\circ\delta={f_0}_{\ast}\delta$, and then the mapping cone also gives a distinguished $n$-exangle ${}\langle M^{\mr}_f,{f_0}_{\ast}\delta\rangle$ in $(\mathscr{C},\mathbf{E}^{1},\mathfrak{r})$.

This shows that $(1_A, c)$ has a good lift of $f^{\centerdot}$ in $(\mathscr{C},\mathbf{E}^{1},\mathfrak{r})$.
\end{proof}

\begin{lemma}\emph{ \cite[Proposition 4.8]{HLN}}\label{lemma2}
Let $\mathscr{C}$ be an additive category with an auto-equivalence $\Sigma \colon \C \rightarrow \C$, and $\mathbb{E}_{\Sigma} = \mathscr{C}(-,\Sigma-)\colon\C\op\times\C\to\Ab$. If we are given an exact realization $\mathfrak{t}$ of $\mathbb{E}_{\Sigma}$. We define that
 $$X_0\overset{f_0}{\longrightarrow}X_1\overset{f_1}{\longrightarrow}\cdots\overset{f_{n-1}}{\longrightarrow}  X_{n}\overset{f_n} {\longrightarrow} X_{n+1}\overset{\delta}{\longrightarrow}\Sigma {X_0}$$
is an $(n+2)$-angle if and only if  $$\mathfrak{t}(\delta)=[X_0\overset{f_0}{\longrightarrow}X_1\overset{f_1}{\longrightarrow}\cdots\overset{f_{n-1}}{\longrightarrow}  X_{n}\overset{f_n} {\longrightarrow} X_{n+1}].$$
Denote this class of $(n+2)$-angles by $\Theta$. Then $(\mathscr{C},\Sigma,\Theta)$ is an $(n+2)$-angulated category.
\end{lemma}

\textbf{Now we are ready to prove the sufficiency of Theorem \ref{main}.}

\begin{proof} Let $(\mathscr{C}, \mathbb{E}, \mathfrak{s})$ be an $n$-exangulated category, where for any object $ X \in \mathscr{C}$, the morphism $X \rightarrow 0$ is a trivial inflation and the morphism $0 \rightarrow X$ is a trivial deflation.
By Proposition \ref{proposition4}, there exists an auto-equivalence $\Sigma \colon \mathscr{C} \rightarrow \mathscr{C}$. By Lemma \ref{lemma3}, we obtain that $\mathbf{E}^{1}(-,-):= \mathscr{C}(-,\Sigma -)$ is an additive bifunctor. By Proposition \ref{proposition2}, there exists a correspondence $\mathfrak{r}$ which associates an equivalence class $\mathfrak{r}(\varepsilon)$ to any extension $\varepsilon \in \mathbf{E}^{1}(C,A)$ for any objects $A,C \in \mathscr{C}$. Moreover, the correspondence $\mathfrak{r}$ is an exact realization. By Proposition \ref{proposition1}, the triple $(\mathscr{C},\mathbf{E}^{1},\mathfrak{r})$ satisfies (EA1), (EA2) and (EA2)$^{\text{op}}$. By Lemma \ref{lemma2}, $\mathscr{C}$ has a structure of an $(n+2)$-angulated category $(\mathscr{C},\Sigma, \Theta)$, where $\Theta$ is the set of $(n+2)$-angles as defined in Lemma \ref{lemma2}. By Lemma \ref{lemma2}, $\Theta$ is $\mathbf{E}^{1}$-compatible. By Lemma \ref{lemma6} ,
$\mathbf{E}^{1}(X_0,Y_0) \cong \mathbb{E}(X_0,Y_0)$, so $\Theta$ is $\mathbb{E}$-compatible.

This completes the proof of Theorem \ref{main}.
\end{proof}

As a special case of Theorem \ref{main} when $n=1$, we have the following.

\begin{corollary}\emph{\cite[Theorem 3.3]{M}}
Let $(\mathscr{C}, \mathbb{E}, \mathfrak{s})$ be an extriangulated category. Then $\mathscr{C}$ has an $\mathbb{E}$-compatible triangulated structure $(\mathscr{C},T,\mathcal{T})$ if and only if for any object $ X \in \mathscr{C}$, the morphism $X \rightarrow 0$ is an inflation and the morphism $0 \rightarrow X$ is a deflation.
\end{corollary}

\proof This follows from Theorem \ref{main} and Remark \ref{remark}.  \qed

\textbf{Jian He}\\
Department of Mathematics, Nanjing University, 210093 Nanjing, Jiangsu, P. R. China\\
E-mail: \textsf{jianhe30@163.com}\\[0.3cm]
\textbf{Panyue Zhou}\\
College of Mathematics, Hunan Institute of Science and Technology, 414006 Yueyang, Hunan, P. R. China.\\
E-mail: \textsf{panyuezhou@163.com}


\begin{thebibliography}{99}
\bibitem[BT]{BT} P. Bergh and M. Thaule. The axioms for $n$-angulated categories. Algebr. Geom. Topol. 13(4): 2405-2428, 2013.

 \bibitem[GKO]{GKO} C. Geiss, B. Keller, S. Oppermann. $n$-angulated categories.  J. Reine Angew. Math.  675: 101--120, 2013.

\bibitem[HLN]{HLN} M. Herschend, Y. Liu, H. Nakaoka.  $n$-exangulated categories. arXiv: 1709.06689v3, 2017.

\bibitem[HZZ1]{HZZ1} J. Hu, D. Zhang, P. Zhou. Proper classes and Gorensteinness in extriangulated categories. J. Algebra 551: 23--60, 2020.

\bibitem[HZZ2]{HZZ2} J. Hu, D. Zhang, P. Zhou. Two new classes of $n$-exangulated categories. J. Algebra 568: 1--21, 2021.

\bibitem[J]{J} G. Jasso.  $n$-abelian and $n$-exact categories. Math. Z. 283(3--4): 703--759, 2016.

\bibitem[LZ]{LZ}  Y. Liu, P. Zhou. Frobenius $n$-exangulated categories. J. Algebra 559: 161--183, 2020.

\bibitem[M]{M} D. Msapato.   A characterisation of extriangulated categories with triangulated structure. arXiv: 2010.07138v1, 2020.

\bibitem[NP]{NP} H. Nakaoka, Y. Palu.  Extriangulated categories, Hovey twin cotorsion pairs and model structures. Cah. Topol. G\'{e}om. Diff\'{e}r. Cat\'{e}g. 60(2): 117--193, 2019.

\bibitem[ZZ]{ZZ} P. Zhou, B. Zhu. Triangulated quotient categories revisted. J. Algebra 502: 196--232, 2018.


\end{thebibliography}
\end{document}